\documentclass[10pt, a4paper, reqno]{amsart}

\usepackage{amscd, dsfont, amsmath, amstext, amsthm, amssymb, amsopn, amsxtra, amsfonts, setspace, fancyhdr, paralist}

\usepackage[mathcal]{euscript}
\usepackage[english]{babel}

\usepackage{pgfplots}
\pgfplotsset{mystyle/.style={color=red,no marks,line width=0.75pt}}

\makeatletter
\g@addto@macro\normalsize{%
  \setlength\abovedisplayskip{10pt}
  \setlength\belowdisplayskip{10pt}
  \setlength\abovedisplayshortskip{5pt}
  \setlength\belowdisplayshortskip{8pt}
}

\allowdisplaybreaks

\newtheoremstyle{normal}
{5pt}
{5pt}
{\normalfont}
{}
{\bfseries}
{}
{0.4em}
{\bfseries{\thmname{#1}\thmnumber{ #2}.\thmnote{ \hspace{0.5em}(#3)\newline}}}

\newtheoremstyle{kursiv}
{5pt}
{5pt}
{\itshape}
{}
{\bfseries}
{}
{0.4em}
{\bfseries{\thmname{#1}\thmnumber{ #2}.\thmnote{ \hspace{0.5em}(#3)\newline}}}

\theoremstyle{normal}

\newtheorem{thm}{Theorem}
\newtheorem{ex}[thm]{Example}
\newtheorem{cor}[thm]{Corollary}
\newtheorem{lem}[thm]{Lemma}

\usepackage{color}

\newcommand{\dd}{\mathrm{d}}

\renewcommand{\epsilon}{\varepsilon}
\newcommand{\omegapr}{\underline{\omega}}
\usepackage{soul}

\begin{document}

\title{Variations on Barb\u{a}lat's Lemma}

\author{B\'{a}lint Farkas\hspace{0.5pt}\MakeLowercase{$^{\text{1,\,a}}$} and Sven-Ake Wegner\hspace{0.5pt}\MakeLowercase{$^{\text{2}}$}}

\renewcommand{\thefootnote}{}
\hspace{-1000pt}\footnote{\hspace{5.5pt}2010 \emph{Mathematics Subject Classification}: Primary 26A06, Secondary 26A12, 26A16, 26A42, 46E39.}
\hspace{-1000pt}\footnote{\hspace{5.5pt}\emph{Key words and phrases}: Barb\u{a}lat's Lemma, H\"older continuous functions, Sobolev spaces.\vspace{1.6pt}}

\hspace{-1000pt}\footnote{\hspace{0pt}$^{1}$\,Bergische Universit\"at Wuppertal, FB C -- Mathematik, Gau\ss{}stra\ss{}e 20, 42119 Wuppertal, Germany, Phone:\hspace{1.2pt}\hspace{1.2pt}+49\hspace{1.2pt}(0)\linebreak\phantom{x}\hspace{1.2pt}202\hspace{1.2pt}/\hspace{1.2pt}439\hspace{1.2pt}-\hspace{1.2pt}2673, Fax:\hspace{1.2pt}\hspace{1.2pt}+49\hspace{1.2pt}(0)\hspace{1.2pt}202\hspace{1.2pt}/\hspace{1.2pt}439\hspace{1.2pt}-\hspace{1.2pt}3724, E-Mail: farkas@uni-wuppertal.de.\vspace{1.6pt}}

\hspace{-1000pt}\footnote{\hspace{0pt}$^{2}$\,Nazarbayev University, Department of Mathematics, School of Science and Technology, 53 Kabanbay Batyr Ave,\linebreak\phantom{x}\hspace{1.2pt}010000 Astana, Kazakhstan, Phone: +7\hspace{1.2pt}(8)\hspace{1.2pt}7172\hspace{1.2pt}/\hspace{1.2pt}69\hspace{1.2pt}-\hspace{1.2pt}4671, E-Mail: svenake.wegner@nu.edu.kz.\vspace{1.6pt}}

\hspace{-1000pt}\footnote{\hspace{0pt}$^{a}$\,B\'{a}lint Farkas was supported by the Hungarian Research Fund (OTKA 100461).}

\begin{abstract}
	It is not hard to prove that a uniformly continuous real function, whose integral up to infinity exists, vanishes at infinity, and it is  probably little known that 	this statement runs under the name \textquotedblleft{}Barb\u{a}lat's Lemma.\textquotedblright{} In fact, the latter name is frequently used in control theory, where the lemma is used to obtain Lyapunov-like stability theorems for non-linear and non-autonomous systems. Barb\u{a}lat's Lemma is \emph{qualitative} in the sense that it asserts that a function has certain properties, here convergence to zero. Such qualitative statements can typically be proved by \emph{``soft analysis''}, such as indirect proofs. Indeed, in the original 1959 paper by Barb\u{a}lat, the lemma was proved by contradiction and this proof prevails in the control theory textbooks. In this short note we first give a direct, \emph{``hard analyis''} proof of the lemma, yielding \emph{quantitative} results, i.e.{} \emph{rates} of convergence to zero. This proof allows also for immediate generalizations. Finally, we unify three different versions which recently appeared and discuss their relation to the original lemma.
\end{abstract}

\maketitle


\section{Direct proof of Barb\u{a}lat's Lemma}\label{SEC-I}

In 1959, Barb\u{a}lat formalized the intuitive principle that a function whose integral up to infinity exists, and whose oscillation is bounded, needs to be small at infinity.

\smallskip

\begin{thm}\label{THM-1}(Barb\u{a}lat's Lemma \cite[p.~269]{Barbalat}) Suppose that $f\colon[0,\infty)\rightarrow\mathbb{R}$ is uniformly continuous and that $\lim_{t\rightarrow\infty}\int_0^tf(\tau)\dd\tau$ exists. Then $\lim_{t\rightarrow\infty}f(t)=0$ holds.
\end{thm}

\smallskip

Barb\u{a}lat's original proof, as well as its reproductions in textbooks, e.g., Khalil \cite[p.~192]{K}, Popov \cite[p.~211]{Popov} and Slotine, Li \cite[p.~124]{SL}, are by contradiction. Our first aim in this note is to give a direct proof of Theorem \ref{THM-1} which also reveals the essence of the statement and enables us to generalize Barb\u{a}lat's Lemma to vector valued functions without difficulty. This is based on the following two lemmas.

\smallskip

\begin{lem}\label{LEM-1} Let $f:[0,\infty)\to\mathbb{R}$ be a continuous function. We define $S\colon[0,\infty)\rightarrow\mathbb{R}$ via $S(t):=\sup_{s\geqslant{}t}\bigl|\int_t^sf(\tau)\dd\tau\bigr|$ and put $\omegapr(a,b,\delta):=\sup\{|f(x)-f(y)|\; | \; x,\,y\in (a,b) \text{ with } |x-y|\leqslant \delta\}$ for $a<b\leqslant\infty$ and $\delta\geqslant 0$. Then $|f(t)|\leqslant S(t)^{1/2}+\omegapr(t,t+S(t)^{1/2},S(t)^{1/2})$ holds for $t\geqslant 0$.
\end{lem}

\smallskip

\begin{proof} Fix $t\geqslant0$. If $S(t)=0$, then $f(t)=0$ and the assertion follows immediately. Also, if $S(t)=\infty$, then there is nothing to prove. Suppose therefore $0<S(t)<\infty$, put $s=S(t)^{1/2}>0$, and compute
\begin{eqnarray*}
|f(t)|&=&\frac 1s\Bigl|\int_t^{t+s} f(t) \dd\tau\Bigr|\\
&\leqslant& \frac 1s\Bigl|\int_t^{t+s} f(\tau) \dd\tau\Bigr|+\frac 1s\Bigl|\int_t^{t+s} \bigl(f(t) -f(\tau)\bigr)\dd\tau\Bigr|\\
&\leqslant& \frac 1s\Bigl|\int_t^{t+s} f(\tau) \dd\tau\Bigr|+\omegapr(t,t+s,s)\\
&\leqslant& \frac 1s S(t)+{\omegapr}(t,t+s,s) \\
&=& \; S(t)^{1/2}+\omegapr(t,t+S(t)^{1/2},S(t)^{1/2})
\end{eqnarray*}
as desired.
\end{proof}

\smallskip

We recall that given a uniformly continuous function $f\colon[0,\infty)\rightarrow\mathbb{R}$, a function $\omega\colon[0,\infty)\rightarrow\mathbb{R}$ is said to be a \emph{modulus of continuity} for $f$ if $\lim_{t\rightarrow 0}\omega(t)=\omega(0)=0$ and $|f(t)-f(\tau)|\leqslant\omega(|t-\tau|)$  for all $t$, $\tau\in[0,\infty)$.

\smallskip

\begin{lem}\label{LEM-2} Let $f:[0,\infty)\to\mathbb{R}$ be uniformly continuous and let $\omega$ be a modulus of continuity for $f$. Consider $S(t)=\sup_{s\geqslant{}t}|\int_t^sf(\tau)\dd\tau|$. Then we have $|f(t)|\leqslant S(t)^{1/2}+\omega(S(t)^{1/2})$ for all $t\geqslant0$.
\end{lem}
\begin{proof} Let $f$ and $\omega$ be given. We define $\underline{\omega}(t):=\omegapr(0,\infty,t)$, where $\omegapr(0,\infty,\cdot)$ is the function defined in Lemma \ref{LEM-1}. Then, $\underline{\omega}(t)\leqslant\omega(t)$ holds for all $t\geqslant0$. Since $\omegapr(a,b,t)\leqslant\omegapr(0,\infty,t)$ holds for all $a<b\leqslant\infty$, we can use Lemma \ref{LEM-1} to obtain
$$
|f(t)|\leqslant S(t)^{1/2}+{\omegapr}(t,t+S(t)^{1/2},S(t)^{1/2})\leqslant  S(t)^{1/2}+\underline{\omega}(S(t)^{1/2})\leqslant S(t)^{1/2}+\omega(S(t)^{1/2})
$$
as desired.
\end{proof}

We now give the direct proof of Barb\u{a}lat's Lemma.

\medskip

\emph{Proof of Theorem \ref{THM-1}.} Let $f$ be uniformly continuous and $\omega$ be a modulus of continuity for $f$. By the Cauchy criterion for indefinite integrals, we obtain (with a direct proof!) that $S(t)\rightarrow0$ for $t\rightarrow\infty$. Therefore, Lemma \ref{LEM-2} yields $|f(t)|\leqslant S(t)^{1/2}+\omega(S(t)^{1/2})\rightarrow 0$ for $t\rightarrow\infty$.\hfill\qed 

\medskip

Since Lemmas \ref{LEM-1} and \ref{LEM-2} remain true for functions with values in a Banach space $E$ (the proofs are verbatim the same), we immediately obtain the next generalization.

\smallskip

\begin{thm}\label{THM-12} Let $E$ be a Banach space and suppose that $f\colon[0,\infty)\rightarrow E$ is uniformly continuous such that  $\lim_{t\rightarrow\infty}\int_0^tf(\tau)\dd\tau$ exists. Then $\lim_{t\rightarrow\infty}f(t)=0$ holds.\hfill\qed
\end{thm}


\section{Barb\u{a}lat's Lemma in a different context}\label{SEC-II}

We pointed out that all proofs of Barb\u{a}lat's Lemma given in the relevant textbooks are indirect. On the other hand, there appeared recently several \textquotedblleft{}alternative versions\textquotedblright{} in the literature whose proofs, or hints for a proof, are based on direct estimates. Tao \cite[Lemma 1]{Tao} states that $\lim_{t\rightarrow\infty}f(t)=0$ holds, whenever $f\in L^2(0,\infty)$ and $f'\in L^{\infty}(0,\infty)$. Desoer and Vidyasagar \cite[Ex.~1 on p.~237]{DV} indicate that it is enough to require that $f$ and $f'$ are in $L^{2}(0,\infty)$ and Teel \cite[Fact 4]{Teel} notes that in the latter the Lebesgue exponent $2$ can be replaced by $p\in [1,\infty)$. Here, $f'$ can be interpreted in the sense of distributions or, equivalently, in the sense that $f$ is absolutely continuous with the almost everywhere existing derivative being essentially bounded.

\smallskip

Indeed, the three results extend the classical statement, that for $1\leqslant{}p<\infty$ all functions in the Sobolev space $W^{1,p}(0,\infty)$ tend to zero for $t\to \infty$ (see, e.g., Brezis \cite[Corollary 8.9]{Brezis}) to the \textquotedblleft{}mixed Sobolev space\textquotedblright{}
$$
W^{1,p,q}(0,\infty)=\bigl\{f\:|\:f\in L^p(0,\infty) \text{ and } f'\in L^q(0,\infty)\bigr\}
$$
for $p=2$, $q=\infty$, and $p=q\in[1,\infty)$, respectively. Our first aim in this section is to prove the following common generalization of the results of Tao \cite{Tao}, Desoer and Vidyasagar \cite{DV}, and Teel \cite{Teel}.
\medskip

\begin{thm}\label{THM-2} Let $p\in [1,\infty)$ and $q\in(1,\infty]$. Every function $f\in W^{1,p,q}(0,\infty)$ tends to zero at infinity.
\end{thm}

\smallskip

Notice, that our proof below shows that all three alternatives are immediate consequences of the original Barb\u{a}lat Lemma. The latter cannot be applied \emph{a priori}, but in view of the following lemma.

\smallskip

\begin{lem}\label{LEM-3} Let $p\in [1,\infty)$ and $q\in (1,\infty]$ be arbitrary. A function $f\in W^{1,p,q}(0,\infty)$ is bounded and uniformly continuous. More precisely, $f$ is $\frac{q-1}q$-H\"older continuous if $q<\infty$ and Lipschitz-continuous if $q=\infty$.
\end{lem}

\begin{proof} For the proof let $q'\in[1,\infty)$ be such that $1/q'+1/q=1$ holds, where we use the convention $1/\infty=0$. In particular, we read $(q-1)/q=1/q'=1$ if $q=\infty$. By our assumptions we have 
$$
f(y)-f(x)=\int_x^y f'(s)\dd s
$$
for almost every $x,y\in [0,\infty)$. Thus, $f$ can be identified with a continuous function satisfying
$$
	 |f(x)-f(y)|\leqslant \Bigl|\int_x^y f'(s)\dd s\Bigr|\leqslant |x-y|^{1/q'}\|f'\|_q.
$$
Here we used H\"older's inequality for $q,q'$ with $1/q+1/q'=1$, so that $f$ is indeed H\"older continuous with exponent $1/q'=(q-1)/q$. Let $r:=p(q-1)/q$. Then we have $r>0$ and 
$$
\tfrac{\dd}{\dd x}|f(x)|^{r+1}=(r+1)|f(x)|^{r-1}f(x)f'(x)
$$
holds. For $x\in [0,\infty)$ we thus obtain
$$
|f(x)|^{r+1}=|f(0)|^{r+1}+(r+1)\int_0^x|f(x)|^{r-1}f(x)f'(x)\,\dd s
\leqslant |f(0)|^{r+1}+(r+1)\|f\|_p^{r}\cdot \|f'\|_q,
$$ 
where the last step is again an application of H\"older's inequality for $q,q'$ with $1/q+1/q'=1$.
\end{proof}

\smallskip

Lemma \ref{LEM-3} enables us to employ the original Barb\u{a}lat Lemma to prove Theorem \ref{THM-2}.

\bigskip

\emph{Proof of Theorem \ref{THM-2}.}  By Lemma \ref{LEM-3} the function $f$ is bounded and uniformly continuous, hence so is $|f|^p$.  Indeed, we have by the mean-value theorem
		\begin{equation}\label{eq:meanvalue}
	\bigl||f(x)|^p-|f(y)|^p\bigr|\leq 	\sup_{|t|\leqslant \|f\|_\infty}p |t|^{p-1}\cdot |f(x)-f(y)|=p\|f\|_\infty^{p-1}|f(x)-f(y)|.
		\end{equation}
		This inequality implies the asserted uniform continuity of $|f|^p$. By assumption, we can apply Barb\u{a}lat's Lemma and obtain the statement.\hfill\qed

\bigskip

Tao's formulation \cite[3rd paragraph on p.~698]{Tao} might erroneously establish the impression that his alternative \cite[Lemma 1]{Tao} uses a weaker assumption than Barb\u{a}lat's Lemma, but has the same conclusion. Our second aim in this section is to illustrate that this is not the case. The following example of a function which satisfies the assumptions of Barb\u{a}lat's Lemma, but not those of \cite[Lemma 1]{Tao} combines two effects. The difference between Lebesgue and improper Riemann integral on the one hand, and that between uniform continuity and having a bounded derivative on the other.

\smallskip

\begin{ex}\label{EXp} Let $f(x)=0$ for $x\in[0,2)$ and $f(x)=(-1)^nf_n(x)$ for $x\in[n,n+1)$ with $n\geqslant2$ and
$$
f_n(x)=\begin{cases}
\:(x-n)^{\frac{1}{2}},&x\in[n,n+\frac{1}{2}n^{-\frac{1}{3}}),\\
\:(n+n^{-\frac{1}{3}}-x)^{\frac{1}{2}},&x\in[n+\frac{1}{2}n^{-\frac{1}{3}},n+n^{-\frac{1}{3}}),\\
\;0,&x\in[n+n^{-\frac{1}{3}},n+1),
\end{cases}
$$
i.e., $f$ looks as follows.

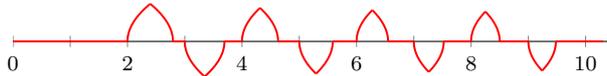
\begin{center}
\begin{tikzpicture}[auto,domain=0:4,x=1cm,y=1cm] 
    \begin{axis}[width=270, height=90,style={font=\footnotesize},
        xmin=0,xmax=10.5,
        ymin=-1,ymax=1,
        axis x line=middle,
        axis y line=none,
        axis line style=->,
        xlabel={},
        ylabel={},
        ytick={},
    xtick={0,1,2,3,4,5,6,7,8,9,10},
    yticklabels={},
    xticklabels={0,,2,,4,,6,,8,,10},
        ]
        \addplot[mystyle] expression[domain=0:2,samples=100]{0} 
                    node[pos=0.65,anchor=south west]{};

        \addplot[mystyle] expression[domain=2:2+1/2*2^(-1/3),samples=100]{sqrt(x-2)} 
                    node[pos=0.65,anchor=south west]{}; 
        \addplot[mystyle] expression[domain=2+1/2*2^(-1/3):2+2^(-1/3),samples=100]{sqrt(2+2^(-1/3)-x)} 
                    node[pos=0.65,anchor=south west]{};
        \addplot[mystyle] expression[domain=2+2^(-1/3):3,samples=100]{0} 
                    node[pos=0.65,anchor=south west]{};

        \addplot[mystyle] expression[domain=3:3+1/2*3^(-1/3),samples=100]{-sqrt(x-3)} 
                    node[pos=0.65,anchor=south west]{}; 
        \addplot[mystyle] expression[domain=3+1/2*3^(-1/3):3+3^(-1/3),samples=100]{-sqrt(3+3^(-1/3)-x)} 
                    node[pos=0.65,anchor=south west]{};
        \addplot[mystyle] expression[domain=3+3^(-1/3):4,samples=100]{0} 
                    node[pos=0.65,anchor=south west]{};

        \addplot[mystyle] expression[domain=4:4+1/2*4^(-1/3),samples=100]{sqrt(x-4)} 
                    node[pos=0.65,anchor=south west]{}; 
        \addplot[mystyle] expression[domain=4+1/2*4^(-1/3):4+4^(-1/3),samples=100]{sqrt(4+4^(-1/3)-x)} 
                    node[pos=0.65,anchor=south west]{};
        \addplot[mystyle] expression[domain=4+4^(-1/3):5,samples=100]{0} 
                    node[pos=0.65,anchor=south west]{};

        \addplot[mystyle] expression[domain=5:5+1/2*5^(-1/3),samples=100]{-sqrt(x-5)} 
                    node[pos=0.65,anchor=south west]{}; 
        \addplot[mystyle] expression[domain=5+1/2*5^(-1/3):5+5^(-1/3),samples=100]{-sqrt(5+5^(-1/3)-x)} 
                    node[pos=0.65,anchor=south west]{};
        \addplot[mystyle] expression[domain=5+5^(-1/3):6,samples=100]{0} 
                    node[pos=0.65,anchor=south west]{};

        \addplot[mystyle] expression[domain=6:6+1/2*6^(-1/3),samples=100]{sqrt(x-6)} 
                    node[pos=0.65,anchor=south west]{}; 
        \addplot[mystyle] expression[domain=6+1/2*6^(-1/3):6+6^(-1/3),samples=100]{sqrt(6+6^(-1/3)-x)} 
                    node[pos=0.65,anchor=south west]{};
        \addplot[mystyle] expression[domain=6+6^(-1/3):7,samples=100]{0} 
                    node[pos=0.65,anchor=south west]{};
                   
        \addplot[mystyle] expression[domain=7:7+1/2*7^(-1/3),samples=100]{-sqrt(x-7)} 
                    node[pos=0.65,anchor=south west]{}; 
        \addplot[mystyle] expression[domain=7+1/2*7^(-1/3):7+7^(-1/3),samples=100]{-sqrt(7+7^(-1/3)-x)} 
                    node[pos=0.65,anchor=south west]{};
        \addplot[mystyle] expression[domain=7+7^(-1/3):8,samples=100]{0} 
                    node[pos=0.65,anchor=south west]{};
                 
        \addplot[mystyle] expression[domain=8:8+1/2*8^(-1/3),samples=100]{sqrt(x-8)} 
                    node[pos=0.65,anchor=south west]{}; 
        \addplot[mystyle] expression[domain=8+1/2*8^(-1/3):8+8^(-1/3),samples=100]{sqrt(8+8^(-1/3)-x)} 
                    node[pos=0.65,anchor=south west]{};
        \addplot[mystyle] expression[domain=8+8^(-1/3):9,samples=100]{0} 
                    node[pos=0.65,anchor=south west]{};  

        \addplot[mystyle] expression[domain=9:9+1/2*9^(-1/3),samples=100]{-sqrt(x-9)} 
                    node[pos=0.65,anchor=south west]{}; 
        \addplot[mystyle] expression[domain=9+1/2*9^(-1/3):9+9^(-1/3),samples=100]{-sqrt(9+9^(-1/3)-x)} 
                    node[pos=0.65,anchor=south west]{};
        \addplot[mystyle] expression[domain=9+9^(-1/3):10.3,samples=100]{0} 
                    node[pos=0.65,anchor=south west]{};       
\end{axis}
\end{tikzpicture}
\end{center}
\noindent{}Straightforward computations show that $\lim_{t\rightarrow\infty}\int_0^tf(\tau)\dd\tau=\frac{\sqrt{2}}{3}\sum_{n=2}^{\infty}(-1)^n\frac{1}{\sqrt{n}}$ exists and that $f$ is uniformly continuous. On the other hand $f\not\in L^2(0,\infty)$ and $f'\not\in L^{\infty}(a,\infty)$ for any $a\geqslant0$.
\end{ex}

\medskip

For given $1\leqslant p<\infty$, the function $f$ in Example \ref{EXp} can easily be modified such that $f\not\in L^p(a,\infty)$ holds. With some additional work, it is also possible to construct a single $f$ such that $f\not\in L^p(a,\infty)$ is true for all $1\leqslant p<\infty$. Finally, in all these cases $f$ can also be changed into a $C^{\infty}$--function; $|f'|$ is then bounded on any finite interval, but unbounded at infinity. Thus, for every $p\in[1,\infty)$ and $q\in(1,\infty]$ there is a function $f$ to which Barb\u{a}lat's Lemma can be applied, but which fails the assumptions of Theorem \ref{THM-2}.

\medskip

Concerning the other direction, it is easy to construct a function $f$ that satisfies the condition of \cite[Lemma 1]{Tao}, i.e., $f\in L^2(0,\infty)$ and $f'\in L^{\infty}(0,\infty)$, but whose improper Riemann integral $\int_0^{\infty}f(t) \dd t$ does not exist. Thus, Tao's alternative is incomparable with the original Barb\u{a}lat Lemma. The same is true for the statement of Theorem \ref{THM-2} whenever $p\not=1$. Only in the case $p=1$, Theorem \ref{THM-2} is a special case of Barb\u{a}lat. Let $f\in W^{1,1,q}$ with $q\in(1,\infty]$. By Lemma \ref{LEM-2}, $f$ is uniformly continuous and bounded. Therefore, $f\in L^1(0,\infty)$ implies that the improper Riemann integral $\int_0^{\infty}f(t)\dd t$ exists.


\bigskip

\section{Rates of convergence}\label{SEC-III}

In this section we use the methods of Section \ref{SEC-I} to derive estimates for the speed of decay in the previous versions of Barb\u{a}lat's Lemma so as to make the statement quantitative. We start with  a modification of Theorem \ref{THM-1}.

\smallskip

\begin{thm}\label{THM-3} Suppose that for $f\colon[0,\infty)\rightarrow\mathbb{R}$ the limit $\lim_{t\rightarrow\infty}\int_0^tf(\tau)\dd\tau$ exists, and that $f$ is H\"older continuous of order $\alpha\in(0,1]$, i.e., $\omega(\tau)=c\tau^{\alpha}$ is a modulus of continuity for a constant $c\geqslant0$.  Then we have $|f(t)|\leqslant(1+c)S(t)^{\alpha/(1+\alpha)}$ for $t\geqslant0$, where $S(t)=\sup_{s\geqslant{}t}|\int_t^sf(\tau)\dd\tau|$.
\end{thm}

\begin{proof} It is enough to repeat the proof of Lemmas \ref{LEM-1}  and \ref{LEM-3} but with $s=S(t)^{1/(1+\alpha)}$.
\end{proof}

\smallskip

Next, we specialize to the situation of Theorem \ref{THM-2}.

\smallskip

\begin{cor}\label{COR-1} Let $f\in W^{1,p,q}(0,\infty)$ for some $p\in [1,\infty)$ and $q\in (1,\infty]$. Then we have $|f(t)|^p\leqslant(1+p\|f\|^{p-1}_{\infty}\|f'\|_{q})S(t)^{(q-1)/(2q-1)}$ for $t\geqslant0$, where $S(t)=\int_t^{\infty}|f(\tau)|^p\dd\tau$.
\end{cor}

\begin{proof} By Lemma \ref{LEM-3}, under our assumptions, $f$ is bounded, and $\tau\mapsto \|f'\|_{q}\tau^{(q-1)/q}$ is a modulus of continuity for $f$. As in equation \eqref{eq:meanvalue} in the proof of Theorem \ref{THM-2}, we conclude that $\omega(\tau)=p\|f\|_\infty^{p-1}\|f'\|_{q}\tau^{(q-1)/q}$ is a modulus of continuity of $|f|^p$. It is therefore enough to apply Theorem \ref{THM-3} to the latter function and to $\alpha={(q-1)}/{q}$ to obtain the assertion, because then $\alpha/(1+\alpha)=(q-1)/(2q-1)$.
\end{proof}

\smallskip

We point out that Corollary \ref{COR-1} contains the quantitative versions of the results of Tao \cite[Lemma 1]{Tao}, Desoer and Vidyasagar \cite[Ex.~1 on p.~237]{DV}, and Teel \cite[Fact 4]{Teel}.

\bigskip

We finish this short note by illustrating how Barb\u{a}lat's Lemma, and its variations are typically used in the control theoretic literature, for instance, to obtain (asymptotic) stability of solutions of ordinary differential equations. The next example is taken from Hou, Duan, Guo \cite[Example 3.1]{HDG}.
\smallskip
\begin{ex}\label{EX} Consider the system
$$
\begin{cases}
\; \dot{e}(t) = -e(t)+\theta(t)\omega(t),\\
\; \dot{\theta}(t) = -e(t)\omega(t),
\end{cases}
$$
with $\omega$ bounded and continuous. For some solution $(e,\theta)$ of this equation we define $V=e^2+\theta^2$ and compute $\dot{V}(t)=-2e^2(t)\leqslant0$ for every $t\geqslant0$. Thus, $V$ and therefore $e$ and $\theta$ are bounded. We put $f:=2e^2$ and in view of the positivity of $V$ and $f$ we obtain
$$
\int_0^Rf(t)\dd{}t \, = \int_0^R-\dot{V}(t)\dd{}t = -V(R)+V(0)\leqslant V(0),
$$
so that $f\in L^1(0,\infty)$.  Since  $\dot{f}(t)=4\dot{e}(t)e(t)=-4e(t)^2+4\theta(t)e(t)\omega(t)$ holds and $\theta,e,\omega$ are bounded, we conclude $\dot{f}\in L^{\infty}(0,\infty)$. That is $f\in W^{1,1,\infty}(0,\infty)$ and by Theorem \ref{THM-2} it follows $f(t)\rightarrow0$ for $t\rightarrow\infty$. From the definition of $f$ we finally obtain that also $e(t)\rightarrow0$ for $t\rightarrow\infty$.
\end{ex}

\medskip

{

\footnotesize

{\sc Acknowledgements. }The authors would like to thank the referees and the editor for their careful work and their valuable comments.

}

\bibliographystyle{amsplain}


\begin{thebibliography}{4}


\bibitem{Barbalat}
I.~Barb{\u{a}}lat, Syst\`emes d'\'equations diff\'erentielles d'oscillations non Lin\'eaires, \emph{Rev. Math. Pures Appl.} \textbf{4} (1959), 267--270.

\bibitem{Brezis}
H.~Brezis, \emph{Functional Analysis, Sobolev Spaces and Partial Differential Equations}, Springer-Verlag, New York, 2011.

\bibitem{DV}
C.~A.~Desoer and M.~Vidyasagar, \emph{Feedback Systems: Input-Output Properties}, Academic Press, New York-London, 1975.

\bibitem{HDG}
M.~Hou, G.~Duan and M.~Guo, New versions of Barb\u{a}lat's lemma with Applications, \emph{J. Control Theory Appl.} \textbf{8} (2010) no.~4,  545--547.

\bibitem{K}
H.~K.~Khalil, \emph{Nonlinear Systems}, Macmillan Publishing Company, New York, 1992.

\bibitem{Popov}
V.~M.~Popov, \emph{Hyperstability of Control Systems}, Springer-Verlag, New York, 1973.

\bibitem{Tao}
G.~Tao, A simple alternative to the Barb\u{a}lat Lemma, \emph{IEEE Trans. Automat. Control} \textbf{42} (1997) no.~5,   698.

\bibitem{Teel}
A.~R.~Teel, Asymptotic convergence from $\mathcal{L}_p$ Stability, \emph{IEEE Trans. Automat. Control} \textbf{44} (1999) no.~11,  2169--2170.

\bibitem{SL}
J.-J.~E.~Slotine and W.~Li, \emph{Applied Nonlinear Control}, Englewood Cliffs, New Jersey, 1991.




\end{thebibliography}

\end{document}